\documentclass[12pt]{article}
\usepackage{graphicx,verbatim,array,multicol,courier}
\usepackage{amsmath,amssymb,amsthm,mathrsfs}
\usepackage[pdftex,bookmarks,colorlinks]{hyperref}
\usepackage{fullpage}
\usepackage{topcapt}
\usepackage{chicago}
\usepackage{color}
\usepackage[section]{placeins} 
\usepackage{mathdots}
\usepackage{mathtools} 
\usepackage{epstopdf} 
\usepackage[labelformat=simple]{subfig}
\usepackage{soul} 
\usepackage{booktabs} 
\usepackage{dsfont}
\usepackage{paralist}
\usepackage[dvipsnames]{xcolor}
\usepackage{enumitem}

\definecolor{navy}{rgb}{0,0,0.502}

\hypersetup{colorlinks,%
citecolor=blue,%
filecolor=green,%
linkcolor=red,%
urlcolor=magenta,%
}

\newtheorem{lemma}{Lemma}[section]

\newtheorem{theorem}[lemma]{Theorem}
\newtheorem{coro}[lemma]{Corollary}
\newtheorem{prop}[lemma]{Proposition}
\newtheorem{exam}[lemma]{\normalfont \scshape
 Example}

\newenvironment{ex}{\begin{exam}}{\end{exam}}

\newcommand{\R}{\mathbb{R}}
\newcommand{\N}{\mathbb{N}}

\newcommand{\abs}[1]{\left\vert#1\right\vert}
\newcommand{\set}[1]{\left\{#1\right\}}

\newcommand{\bfx}{\boldsymbol{x}}

\newcommand{\bfU}{\boldsymbol{U}}
\newcommand{\bfu}{\boldsymbol{u}}

\newcommand{\bfX}{{\boldsymbol{X}}}

\newcommand{\bfy}{\boldsymbol{y}}

\newcommand{\bfeta}{\boldsymbol{\eta}}

\DeclareMathOperator{\expect}{E}

\DeclareMathOperator{\comp}{c}

\DeclareMathOperator{\Pro}{Pr}

\def\indic{\mathds{1}}

\def\diff{\mbox{d}}

\newcommand{\iscrec}{I}

\allowdisplaybreaks[4]

\begin{document}

\title{On Multivariate Records from Random Vectors with Independent Components}%
\author{M. Falk, A. Khorrami and S. A. Padoan}

\maketitle

\begin{abstract}
Let $\bfX_1,\bfX_2,\dots$ be independent copies of a random vector $\bfX$ with values in
$\R^d$ and with a continuous distribution function. 
The random vector $\bfX_n$ is a complete record, each of its components is a record.
As we require $\bfX$ to have independent components, crucial results for univariate records clearly carry over. But there are substantial differences as well: While there are infinitely many records in case $d=1$, there occur only finitely many in the series if $d\geq 2$.
Consequently, there is a terminal complete record with probability one. We compute the distribution of the random total number of complete records and investigate the distribution of the terminal record. 
For complete records, the sequence of waiting times  forms a Markov chain, but differently from the univariate case, now the state infinity is an absorbing element of the state space
\end{abstract}

\textbf{Keyword and phrases:} Multivariate records, complete records, terminal record, waiting time, Markov chain.

\section{Introduction}

Let $\bfX_1,\bfX_2,\dots$ be independent copies of a random vector $\bfX\in\R^d$
with distribution function $F$. We assume that the margins $F_i,1\leq i\leq d$, of
$F$ are continuous univariate distribution functions.
This is equivalent of assuming the condition that $F$ itself is a continuous distribution function.


Records among a sequence of univariate independent and identically distributed (iid) random variables $X_1,X_2,\dots$
have been extensively investigated over the past decades, see, e.g.,
\citeN[Ch. 4.1 and 4.2]{resn87}, \citeN[ Ch. 6.2 and 6.3]{gal87}, and \citeN{arnbn98}.

For example, consider the indicator function
\[
e_m:=\indic(X_m \text{ is a record}),\qquad m\in\N,
\]
where $\indic(E)$ denotes the indicator function of the event $E$ and
$X_1,X_2,\dots$ are iid univariate random variables with a joint continuous df $F$. It is well known that the indicator functions $e_1,e_2,\ldots$ are independent with
\begin{equation}\label{eq: prob record}
\Pro(e_m=1)=m^{-1},\qquad m\in\N,
\end{equation}
see, e.g., \citeNP[(Lemma 6.3.3)]{gal87}.

In this paper we are interested in \emph{complete} records. The $d$-dimensional random vector (rv)
$\bfX_j$ is a complete record if each of its components is a record, i.e.,
\[
\bfX_j>\max_{1\leq i\leq j-1}\bfX_i,
\]
where the maximum is taken componentwise. All our operations on vectors $\bfx=(x_1,\dots,x_d)$, $\bfy=(y_1\dots,y_d)$, such as $\bfx<\bfy$, are meant componentwise.
Clearly, $\bfX_1$ is a complete record.

Multivariate records have not been discussed that extensively, yet they have been approached by, \citeN{golres89}, \citeN{golres95} and \citeN[Chapter 8]{arnbn98}, for instance.
For supplementary material on multivariate and functional records we refer to the thesis by \citeN{zott16} and the references cited therein.

In this paper we focus on the case, where the components of $\bfX$ are independent.
Then, clearly, various of the results for univariate vectors carry over to the multivariate case.
In particular the preceding result carries over: Put
\[
I_m:=\indic(\bfX_m\text{ is a complete record}).
\]
Then, the indicator functions $I_1,I_2,\dots$ are independent with
\begin{equation}\label{eq: prob complete}
\Pro(I_m=1)= m^{-d},\qquad m\in\N.
\end{equation}
However, from equations \eqref{eq: prob record} and \eqref{eq: prob complete} we immediately deduce a first difference between the theory of univariate and multivariate records. If the joint distribution function of the sequence of iid univariate random variables is continuous, then the total number of records in this series is infinite with probability one.
On the other hand, by equation \eqref{eq: prob complete} we have
\begin{equation}\label{eq: CR finite}
\expect\left(\sum_{m\in\N}I_m\right)=\sum_{m\in\N}\expect(I_m)=\sum_{m\in\N}m^{-d} < \infty
\end{equation}
if $d\geq 2$.
As a consequence, the total number of complete records $\sum_{m\in\N}I_m$ is finite with probability one. Hence, in case $d\geq 2$, there is a terminal complete record in the series $\bfX_1,\bfX_2,\dots$
In Section \ref{sec: terminal record} we compute the distribution of the random total number of complete records and we investigate the distribution of the terminal record.
In Section \ref{sec:rec_time} we study the sequence of waiting times for the complete records. Such a sequence forms a Markov chain, similarly to the univariate case, but in higher dimensions the state infinity is an absorbing element of this state space.

Suppose that the components of $\bfX$ are not independent, but that its distribution function is in the max-domain of attraction of a multivariate extreme value distribution. \citeN[Theorem 5.3]{golres89} proved in case $d=2$ that the total number of complete records is finite if and only if the limiting extreme value  distribution has independent components. Assuming that the components of $\bfX$ are independent, we only require continuity of its distribution function, we do not requite that it is in the max-domain of attraction of an extreme value distribution.
\section{Terminal record}\label{sec: terminal record}
Let
\[
T:= \sup\set{m\in\N:\iscrec_m=1},
\]
which is the index of the ultimate complete record in the sequence $\bfX_1,\bfX_2,\dots$
If $d\geq 2$, then we know from equation \eqref{eq: CR finite} that $\Pro(T<\infty)=1$. In the next Lemma we compute the distribution of $T$.
\begin{lemma}\label{lemma: p_k formulas}
For $d\geq 2$ and $k\in\N$
\begin{equation}\label{eq: p_k}
p_k = \Pro(T=k)=\frac 1{k^d} \prod_{m\ge k+1}\left(1-\frac 1{m^d} \right).
\end{equation}
In particular, when $d=2$:
\[
p_k=\frac{1}{k(k+1)}.
\]
\end{lemma}
\begin{proof}
The independence of the indicator functions $I_m,\,m\in\N$, together with equation \eqref{eq: prob complete} imply the first assertion.

In the case $d=2$ we obtain
\[
p_k= \frac 1{k^2} \lim_{N\to\infty}\prod_{m= k+1}^N \frac{(m-1)(m+1)}{m^2}= \frac 1{k^2}
\lim_{N\to\infty}\frac{k(N+1)}{N(k+1)}
= \frac{1}{k(k+1)}.
\]
\end{proof}
The first observation $\bfX_1$ is a record by definition.
By the preceding result, in dimension $d=2$, $\bfX_1$ is already the terminal complete record with probability $p_1=1/2$.
The next observation $\bfX_2$ is with probability $1/4$ a complete record, and it is with probability $1/6$ the terminal complete record.
From equation \eqref{eq: CR finite} and Lemma \ref{lemma: p_k formulas} we have that
\[
1=\Pro(T<\infty)= \sum_{k\in\N}p_k = \sum_{k\in\N}\frac 1{k^d} \prod_{m\ge k+1}\left(1-\frac 1{m^d} \right)
\]
which, taken as a purely mathematical formula, is a nice by-product.

The probability $p_1=p_1(d)$ increases as the dimension increases, whereas $p_k=p_k(d),\,k\geq 2$, decreases. This is the content of the next Lemma.
\begin{lemma}\label{lemma: lim p_k}
We have
\[
\lim_{d\rightarrow\infty}p_1(d)=\lim_{d\rightarrow\infty}\prod_{m\ge 2}\left(1-\frac 1{m^d} \right)=1,
\]
whereas, for $k\geq 2$, we have
$$
\lim_{d\rightarrow\infty}p_k(d)=\lim_{d\rightarrow\infty}\frac 1{k^d} \prod_{m\ge k+1}\left(1-\frac 1{m^d} \right)=0.
$$
\end{lemma}
\begin{proof}
The second assertion is immediate from the bound $\prod_{m\ge k+1}\left(1-\frac 1{m^d} \right)\leq 1$. The first assertion is a consequence of the equation
$\sum_{k\in\N}p_k=1$ and the following bound, valid for $d\geq 4$,
\[
\sum_{k\geq 2}p_k=\sum_{k\geq 2}\frac 1{k^d} \prod_{m\ge k+1}\left(1-\frac 1{m^d} \right)\leq \frac{1}{2^{d-2}}\sum_{k\geq 2}\frac 1{k^2}\xrightarrow[d\to\infty]{}0.
\]
\end{proof}
Lemma \ref{lemma: p_k formulas} implies that
the expected arrival time for the final complete record is
\[
\expect(T) = \sum_{k\in\N}kp_k = \sum_{k\in\N} \frac 1{k^{d-1}} \prod_{m\ge k+1} \left(1-\frac 1{m^d}\right).
\]
Therefore,  we have
\[
\expect(T)=\sum_{k\in\N}\frac{1}{k+1}=\infty
\]
when $d=2$, while
\[
\expect(T)\leq\sum_{k\in\N}\frac{1}{k^{d-1}}<\infty
\]
for $d\geq 3$.

Finally, by repeating the arguments in the proof of Lemma \ref{lemma: lim p_k}, we have
\[
\expect(T)\xrightarrow[d\to\infty]{}1.
\]

Let $\bfeta$ be a $d$-dimensional rv with independent components $\eta_1,\dots,\eta_d$, each following a standard negative exponential distribution, i.e.,
$\Pro(\eta_i\leq x)=\exp(x)$, $x\leq 0$, for all $i\leq d$.
In what follows we investigate the distribution of the terminal record, i.e.,
we study $\Pro(\bfeta_T\leq\bfx)$,
where $T$ denotes again the random index of the terminal record and $\bfx=(x_1,\dots,x_d)\in {(-\infty,0]}^d$.
We have a closed formula for  $\Pro(\bfeta_T\leq\bfx)$, see Theorem \ref{teo: final complete record}.
%
However, we first want to verify the following conjecture:
Let $T=T(d)$ be the random index of the terminal record which depends on the dimension $d$.
From Lemma \ref{lemma: lim p_k} we know that $\Pro(T(d)=1)=p_1(d)\xrightarrow[d\to\infty]{}1$.
Therefore, one would expect that
\[
\Pro(\bfeta_T\leq\bfx)\approx\Pro(\bfeta_1\leq\bfx)=\exp\left(\sum_{i=1}^dx_i\right),
\]
when $d$ gets large. This conjecture is verified in the next result.
To ease the notation we drop the dependence on $d$, wherever it causes no ambiguities.
\begin{prop}\label{prop: T rec approx}
Let $x_1,x_2,\dots$ be a sequence of numbers in $(-\infty,0]$.
\begin{enumerate}[label=(\roman*)]
\item If
$$
\sum_{i=1}^\infty x_i\in (-\infty,0],
$$
then, with  $\bfx_d:=(x_1,\dots,x_d)$,  $d\in\N$, we have
\[
\Pro(\bfeta_T\leq\bfx_d)\xrightarrow[d\to\infty]{}\exp\left(\sum_{i=1}^\infty x_i\right)=\lim_{d\to\infty}\Pro(\bfeta_1\leq\bfx_d).
\]
\item If
$$
\lim_{d\to\infty}\sum_{i=1}^dx_i=-\infty,
$$
but such that
\begin{equation}\label{cond: T rec approx}
\limsup_{d\to\infty}\abs{\frac{1}{d}\sum_{i=1}^dx_i}<\log 2
\end{equation}
then
\[
\frac{\Pro(\bfeta_T\leq\bfx_d)}{\Pro(\bfeta_1\leq\bfx_d)}\xrightarrow[d\to\infty]{}1.
\]
\end{enumerate}
\end{prop}
\begin{proof}
We have
\[
\Pro(\bfeta_T\leq\bfx_d)=\Pro(\bfeta_T\leq\bfx_d, T=1)+\Pro(\bfeta_T\leq\bfx_d, T\geq 2)
\]
with
\begin{equation}\label{eq: bound 1}
\begin{split}
\Pro(\bfeta_T\leq\bfx_d, T\geq 2)&=\sum_{k=2}^\infty \Pro(\bfeta_T\leq\bfx_d, T=k)\\
&\leq \sum_{k=2}^\infty \Pro(\bfeta_k\leq\bfx_d, \bfeta_k\text{ is a complete record})\\
&= \sum_{k=2}^\infty \Pro(\bfeta_k\leq\bfx_d\mid \bfeta_k\text{ is a complete record})\frac{1}{k^d}\\
&= \sum_{k=2}^\infty \exp\left(k\sum_{i=1}^dx_i\right)\frac{1}{k^d}\\
&= \exp\left(\sum_{i=1}^dx_i\right)\sum_{k=2}^\infty \exp\left((k-1)\sum_{i=1}^dx_i\right)\frac{1}{k^d}\\
&= \exp\left(\sum_{i=1}^dx_i\right)o(1).
\end{split}
\end{equation}
%
In the preceding list we used the fact that the \emph{univariate} distribution function $\Pro(\eta_k\le x\mid \eta_k \mbox{ is a record})$ equals $\exp(kx)$, $x\le 0$, $k\in\N$, as established in
\citeN{falk+k+p17}.

From equation \eqref{eq: bound 1} we obtain
\[
\begin{split}
\Pro(\bfeta_T\leq\bfx_d)&=\Pro(\bfeta_T\leq\bfx_d, T=1)+\exp\left(\sum_{i=1}^dx_i\right)o(1)\\
&=\Pro(\bfeta_1\leq\bfx_d)-\Pro(\bfeta_1\leq\bfx_d, T\geq 2)+\exp\left(\sum_{i=1}^dx_i\right)o(1)\\
&=\Pro(\bfeta_1\leq\bfx_d)-\Pro(\bfeta_1\leq\bfx_d\mid T\geq 2)\Pro(T\geq 2)+\exp\left(\sum_{i=1}^dx_i\right)o(1)
\end{split}
\]
As $\Pro(T\geq 2)\xrightarrow[d\to\infty]{}0$, this implies the first assertion.

Next, suppose that
$$
\sum_{i=1}^\infty x_i=-\infty.
$$
We have to show that
\begin{equation}\label{eq: bound 2}
\Pro(\bfeta_1\leq\bfx_d, T\geq 2)= \exp\left(\sum_{i=1}^dx_i\right)o(1)
\end{equation}
as well.
Hoelder's inequality implies with $p,q\geq 1,\,p^{-1}+q^{-1}=1$,
\[
\begin{split}
\Pro(\bfeta_T\leq\bfx_d,T\geq 2)&=\expect\left(\indic(\bfeta_1\leq \bfx_d)\indic(T\geq 2)\right)\\
&\leq{\Pro(\bfeta_1\leq\bfx_d)}^{1/p}{\Pro(T\geq 2)}^{1/q}\\
&=\exp\left(\frac{1}{p}\sum_{i=1}^dx_i\right){\Pro(T\geq 2)}^{1/q}
\end{split}
\]
where
\[
\begin{split}
\Pro(T\geq 2)&=\Pro(T= 2)+\sum_{k=3}^\infty \Pro(T=k)\\
&\leq 2^{-d}+\sum_{k=3}^\infty k^{-d}\\
&\leq 2^{-d}+\int_2^\infty x^{-d}\diff x\\
&=2^{-d}+\frac{2^{-d+1}}{d-1}\leq\frac{3}{2^d}.
\end{split}
\]
Thus, we obtain
\[
\begin{split}
\Pro(\bfeta_1\leq\bfx_d,T\geq 2)&\leq\exp\left(\frac{1}{p}\sum_{i=1}^dx_i\right)\frac{3^{1/q}}{2^{d/p}}\\
&=\exp\left(\sum_{i=1}^dx_i\right)\exp\left(\left(\frac{1}{p}-1\right) \sum_{i=1}^dx_i-\frac{d}{q}\log 2\right)3^{1/q}\\
&=\exp\left(\sum_{i=1}^dx_i\right)\exp\left(-\frac{1}{q}\sum_{i=1}^dx_i-\frac{d}{q}\log 2\right)3^{1/q}\\
&=\exp\left(\sum_{i=1}^dx_i\right)\exp\left(\frac{d}{q}\left(-\frac{1}{d}\sum_{i=1}^dx_i-\log 2\right)\right)3^{1/q}\\
&= \exp\left(\sum_{i=1}^dx_i\right)o(1)
\end{split}
\]
by condition \eqref{eq: CR finite}. This proves the second assertion as well.
\end{proof}
By assuming the componentwise representation $X_i=F^{-1}_i(\exp(\eta_i))$, $i=1,\ldots,d$,
$d\in\N$, for each component $i=1,\dots,d$,
the preceding result immediately carries over to a sequence $\bfX_1,\bfX_2,\dots$ of independent copies of a rv $\bfX$ with
independent components and univariate continuous marginal df $F_1,\dots,F_d$.
\begin{coro}
Let $y_1,y_2,\dots$ be a sequence of numbers in $\R$.
\begin{enumerate}[label=(\roman*)]
\item If
$$
\prod_{i=1}^\infty F_i(y_i)\in (0,1),
$$
then
\[
\Pro(\bfX_T\leq\bfy_d)\xrightarrow[d\to\infty]{}\prod_{i=1}^\infty F_i(y_i)=\lim_{d\to\infty}\Pro(\bfX_1\leq\bfy_d).
\]
\item If
$$
\prod_{i=1}^\infty F_i(y_i)=0,
$$
but such that
\[
\liminf_{d\to\infty}{\left(\prod_{i=1}^d F_i(y_i)\right)}^{1/d}>\frac{1}{2},
\]
then
\[
\frac{\Pro(\bfX_T\leq\bfy_d)}{\Pro(\bfX_1\leq\bfy_d)}\xrightarrow[d\to\infty]{}1.
\]
\end{enumerate}
\end{coro}
In the final result of this section we derive the exact distribution of the terminal complete record for fixed dimension $d\geq 2$.
We suppose again a sequence $\bfX_1,\bfX_2,\dots$ of independent copies of a rv $\bfX\in\R^d$
with independent components and continuous univariate marginal df $F_1,\dots,F_d$.
\begin{theorem}\label{teo: final complete record}
The distribution function of the final complete record is
\begin{align*}\label{eq: final complete record}
\Pr\left(\bfX_{T}\leq \bfx_d\right)=
\sum_{k=1}^{\infty}\left(\frac{\prod_{i=1}^d u_i^{k}}{{k}^d}
-\sum_{K\subseteq \mathcal{J}}{(-1)}^{\abs{K}-1}\prod_{i=1}^d
\left(\sum_{r\in K'} \frac{u_i^{r}}{r\prod_{s\neq r \in K'}(s-r)}+\frac{1}{\prod_{r\in K'} r}\right)\right),
\end{align*}
where $u_i=F_i(x_i)$, $\mathcal{J}=\{k+1,k+2,\dots\}$, $K\subseteq\mathcal{J}$,
$K'=\{k\}\cup K$ and $\abs{K}$ is the total number of elements in the set $K$.
\end{theorem}
\begin{proof}[Proof of Theorem~\ref{teo: final complete record}]
Without loss of generality we provide the proof with uniform margins $U_1,U_2\ldots$
We look for the solution of
\[
\Pro\left(\bfU_T\leq\bfu_d\right)=\sum_{k=1}^\infty \Pro\left(\bfU_k\leq\bfu_d\mid\iscrec_k=1,\bigcap_{m=k+1}^\infty\{\iscrec_m=0\}\right)\Pro(T=k),
\]
and $\Pro(T=k)=p_k$.
The probability of the conditioning event is given by \eqref{eq: p_k}, therefore we only need to compute
\[
\begin{split}
&\Pro\left(\bfU_k\leq\bfu_d,\iscrec_k=1,\bigcap_{m=k+1}^\infty\{\iscrec_m=0\}\right)\\
&=\Pro\left(\bfU_k\leq\bfu_d,\iscrec_k=1\right)-\Pro\left(\bfU_k\leq\bfu_d,\iscrec_k=1,{\left(\bigcap_{m=k+1}^\infty\{\iscrec_m=0\}\right)}^{\!\!\comp}\right)
\end{split}
\]
Since the components of $\bfU$ are independent, it is easy to see that
\[
\Pro\left(\bfU_k\leq\bfu_d,\iscrec_k=1\right)=\frac{\prod_{i=1}^d u_i^k}{k^d}.
\]
By means of the inclusion-exclusion principle, we have that
\[
\begin{split}
&\Pro\left(\bfU_k\leq\bfu_d,\iscrec_k=1,{\left(\bigcap_{m=k+1}^\infty\{\iscrec_m=0\}\right)}^{\!\!\comp}\right)\\
&=\sum_{K\subseteq\mathcal{J}}{(-1)}^{\abs{K}-1}\Pro\left(\bfU_k\leq\bfu_d,\iscrec_k=1,\iscrec_t=1,t\in K\right)\\
&=\sum_{K\subseteq\mathcal{J}}{(-1)}^{\abs{K}-1}\prod_{i=1}^d\Pro\left(U_{k,i}\leq u_i,I_k=1,I_t=1,t\in K\right)
\end{split}
\]
where $K=\{j_1,\dots,j_{|K|}\}\subseteq\mathcal{J}=\{k+1,k+2,\dots\}$. Note that
\[
\Pro\left(U_{k,i}\leq u_i,e_k=1,e_t=1,t\in K\right)=\underset{0<z\leq \min{(u_i,z_1)}<z_2\leq\dots\leq z_{\abs{K}}\leq 1}{\int\dots\int}z^{k-1}\prod_{t=1}^{\abs{K}}z_t^{j_t-j_{t-1}-1}\,\diff z\,\diff z_1\dots\diff z_{\abs{K}}.
\]
We compute the previous probability by using the induction principle. We claim that
\[
\begin{split}
A_m&=\underset{0<z\leq \min{(u_i,z_1)}<z_2\leq\dots\leq z_m}{\int\dots\int}z^{k-1}\prod_{t=1}^{m-1}z_t^{j_t-j_{t-1}-1}\,\diff z\,\diff z_1\dots\diff z_{\abs{K}-1}\\
&=\sum_{t=0}^{m-1}{(-1)}^t\frac{u_i^{j_t}}{\prod_{r=0}^{t-1}(j_t-j_r)j_t\prod_{r=t+1}^{m-1}(j_r-j_t)}z_m^{j_{m-1}-j_t}\indic(u_i<z_m)+\frac{z_m^{j_{m-1}}}{\prod_{t=0}^{m-1}j_t}\indic(u_i>z_m),
\end{split}
\]
where $j_0=k$ and the products in the denominator in the second row of the term on the left-hand side are equal one by convention, whenever $t-1<0$ or $m-1<t+1$. At the first step we have
\[
A_1=\int_0^{\min(u_i,z_1)}z^{k-1}\diff z=\frac{u_i^k}{k}\indic\left(u_i\leq z_1\right)+\frac{z_1^k}{k}\indic\left(u_i> z_1\right).
\]
Now, let us suppose the claim is true for $m-1$, i.e.
\[
A_{m-1}=\sum_{t=0}^{m-2}{(-1)}^t\frac{u_i^{j_t}}{\prod_{r=0}^{t-1}(j_t-j_r)j_t\prod_{r=t+1}^{m-2}(j_r-j_t)}z_{m-1}^{j_{m-2}-j_t}\indic(u_i<z_{m-1})+\frac{z_{m-1}^{j_{m-2}}}{\prod_{t=0}^{m-2}j_t}\indic(u_i>z_{m-1}),
\]
and prove it for $m$.
\[
\begin{split}
A_m&=\int_0^{z_m}z_{m-1}^{j_{m-1}-j_{m-2}-1}A_{m-1}\diff z_{m-1}\\
&=\sum_{t=0}^{m-2}{(-1)}^t\frac{u_i^{j_t}}{\prod_{r=0}^{t-1}(j_t-j_r)j_t\prod_{r=t+1}^{m-2}(j_r-j_t)}\frac{z_m^{j_{m-1}-j_t}-u_i^{j_{m-1}-j_t}}{j_{m-1}-j_t} \indic(u_i<z_m)\\
&+\frac{1}{\prod_{t=0}^{m-2}j_t}\left(\int_0^{u_i} z_{m-1}^{j_{m-1}-1}\diff z_{m-1}\indic(u_i<z_{m})+\int_0^{z_m} z_{m-1}^{j_{m-1}-1}\diff z_{m-1}\indic(u_i>z_{m})\right)
\end{split}
\]
which proves the claim. By considering $z_m=1$ and by noting that $\prod_{r=0}^{t-1}(j_t-j_r)={(-1)}^t\prod_{r=0}^{t-1}(j_r-j_t)$ and substituting with $u_i=F_i(x_i)$, the proof is complete.
\end{proof}
\section{Complete Record Times}\label{sec:rec_time}
In this section we derive some results on record times. Let
\begin{equation}\label{def: rec time}
R(n):= \inf\set{m\in\N:\, \sum_{i=1}^m \iscrec_i=n}, \quad n\ge 2,\quad R(1):=1,
\end{equation}
be the arrival time of the $n$-th complete record, where
 $\inf\emptyset:=\infty$, which describes the case when there is no further complete record. We have seen in equation \eqref{eq: CR finite} that the number of complete records is finite with probability one, if the dimension $d$ of our observations is at least $2$. We start with the exact distribution of $R(n)$. Note that the distribution of $R(n)$ does not depend on the underlying univariate df $F_1,\dots,F_d$, provided that they are continuous.
\begin{prop}\label{prop: R(n) cr indep}
For a generic size $d\geq 2$ we have
\begin{equation}\label{eq: pois_bin}
\Pro(R(n)=k) = k^{-d} \sum_{A\subseteq \{2,\dots,k-1\},\abs{A}=n-2}\prod_{q\in A} q^{-d}
\prod_{m\in A^{\comp}}\left(1-m^{-d} \right),\qquad k\ge n,
\end{equation}
where $\abs{A}$ is the total number of elements in the set $A$.
\end{prop}
\begin{proof}
Note that
\[
\Pro(R(n)=k)=\Pro\left(\iscrec_k=1,S_{k-1}=n-1\right)=
\Pro\left(\iscrec_k=1\right)\Pro\left(S_{k-1}=n-1\right),
\]
where $S_{k-1}=\sum_{m=1}^{k-1}\iscrec_m$ is a sum of independent Bernoulli random variables, each with parameter $m^{-d}$.
Therefore,
\[
\Pro\left(S_{k-1}=n-1\right)=\Pro\left(\sum_{m=2}^{k-1} \iscrec_m=n-2\right)
=
\sum_{A\in \mathcal{A}_{n-2}}\prod_{q\in A}q^{-d}\prod_{m\in A^{\comp}}\left(1-m^{-d} \right),
\]
which is a Poisson-Binomial distribution.
\end{proof}
\begin{ex}\upshape
For $n=2$ we get
\[
\Pro(R(2)=k) = \frac 1{k^d} \prod_{j=2}^{k-1}\left(1-\frac 1{j^d} \right),\qquad k\ge 2
\]
while if $n=3$
\[
\begin{split}
\Pro(R(3)=k)
=\frac 1{k^d} \prod_{j=2}^{k-1}\left(1-\frac 1{j^d} \right)\sum_{i=2}^{k-1}\frac{1}{i^d-1},\qquad k\ge 3.
\end{split}
\]
Thus in the special case $d=2$ we obtain
\begin{equation}\label{eqn:examples_times}
\Pro(R(2)=k)
= \frac 1 {2k(k-1)},\qquad
\Pro(R(3)=k)
=\frac{3k^2-7k+2}{8k^2{(k-1)}^2}.
\end{equation}
\end{ex}
The sequence $R(n),\,n\geq 2$, is a Markov chain, as it is in the univariate case, see, e.g., \citeNP[(Section 6.3)]{gal87}. Note that the state space is now $\{2,3,\dots\}\cup\{\infty\}$.
\begin{prop}\label{prop: T(n) cr indep}
The sequence $R(n)$, $n\geq 2$ forms a Markov chain with the following transition probabilities
\begin{equation}\label{eq: transitions R(n)}
\Pro\left(R(n)=k\vert R(n-1)=j\right)=
\begin{cases}
k^{-d}, & for\; k=j+1,\\
k^{-d}\prod_{m=j+1}^{k-1}\left(1-m^{-d}\right), & for \; k>j+1,\\
\prod_{m=j+1}^{\infty}\left(1-m^{-d}\right), & for \; k=\infty >j,
\end{cases}
\end{equation}
with $j\geq n-1$.
The state $\{\infty\}$ is absorbing, that is $\Pro(R(n)=\infty\mid R(n-1)=\infty)=1$, when
$n\geq 3$.
\end{prop}
\begin{proof}
For a finite sequence of finite states,
by the independence of
$\iscrec_1,\iscrec_2,\ldots$ we have
\begin{equation}\label{eq:time_joint}
\Pro\left(R(m)=j_m, 2\leq m\leq n\right)
=\prod_{m=2}^n\Pro\left(\iscrec_{j_m}=1\right)\Pro\left(\sum_{i=j_{m-1}+1}^{j_m-1}\iscrec_i=0\right),
\end{equation}
where $j_1=2$ by convention. Using this formula we obtain for the conditional probability
\begin{equation}\label{eq:time_cond}
\Pro\left(R(n)=j_n\vert R(m)=j_m,2\leq m\leq n-1\right)=\Pro\left(\iscrec_{j_n}=1\right)
\Pro\left(\sum_{i=j_{n-1}+1}^{j_n-1}\iscrec_i=0\right).
\end{equation}
Note that
\[
\Pro\left(R(n-1)=j_{n-1}\right)=\Pro\left(\iscrec_{j_{n-1}}=1\right)\Pro\left(\sum_{i=2}^{j_{n-1}-1}\iscrec_i=n-2\right)
\]
and
\begin{align*}
\Pro\left(R(n)=j_n, R(n-1)=j_{n-1}\right)=&
\Pro\left(\iscrec_{j_n}=1\right)\Pro\left(\sum_{i=j_{n-1}+1}^{j_n-1}\iscrec_i=0\right)\\
&\Pro\left(\iscrec_{j_{n-1}}=1\right)\Pro\left(\sum_{i=2}^{j_{n-1}-1}\iscrec_i=n-2\right).
\end{align*}
Therefore $\Pro\left(R(n)=j_n\vert R(n-1)=j_{n-1}\right)$ is equal to right hand-side of \eqref{eq:time_cond}.

For the case that a time moves from a finite state to infinity we have
\[
\begin{split}
\Pro\left(R(n)=\infty, R(m)=j_m, 2\leq m\leq n-1\right)
&=\prod_{m=2}^{n-1}\Pro\left(\iscrec_{j_m}=1\right)\\
&\times\Pro\left(\sum_{i=j_{m-1}+1}^{j_m-1}\iscrec_i=0\right)\\
&\times\Pro\left(\sum_{i=j_m+1}^{\infty}\iscrec_i=0\right).
\end{split}
\]
Using this result and the one in \eqref{eq:time_joint} we obtain
\begin{equation}\label{eq:time_inf_cond}
\Pro\left(R(n)=\infty\vert R(m)=j_m,2\leq m\leq n-1\right)=\Pro\left(\sum_{i=j_{n-1}+1}^{\infty}\iscrec_i=0\right).
\end{equation}
Now, noting that $\Pro\left(R(n-1)=j_{n-1}\right)=\Pro\left(\iscrec_{j_{n-1}}=1\right)$ and
$$
\Pro\left(R(n)=\infty, R(n-1)=j_{n-1}\right)=\Pro\left(\iscrec_{j_{n-1}}=1\right)\Pro\left(\sum_{i=j_{n-1}+1}^{\infty}\iscrec_i=0\right),
$$
then $\Pro\left(R(n)=\infty\vert R(n-1)=j_{n-1}\right)$ is equal to right hand-side of \eqref{eq:time_inf_cond}.

To compute the transition probabilities, note that $\Pro(\iscrec_n=1)=n^{-d}$ and $\Pro(\iscrec_n=0)=1-n^{-d}$.
Finally, to complete the proof we need to check that
\[
p_{n|n-1}=\sum_{k\geq j+1}\Pro\left(R(n)=k\vert R(n-1)=j\right)+\Pro\left(R(n)=\infty\vert R(n-1)=j\right)=1,
\]
for each $j\geq 2$, i.e.
\[
{(j+1)}^{-d}+\sum_{k=j+2}^{\infty} k^{-d}\prod_{m=j+1}^{k-1}(1-m^{-d})+\prod_{m=j+1}^{\infty}(1-m^{-d})=1.
\]
We prove by induction that
\begin{equation}\label{eq: sum tp=1}
\sum_{k=j+2}^M k^{-d}\prod_{m=j+1}^{k-1}(1-m^{-d})+\prod_{m=j+1}^M(1-m^{-d})=1-{(j+1)}^{-d},
\end{equation}
for each $M\in\N$.

Step $M=j+2$:
\[
{(j+2)}^{-d}(1-{(j+1)}^{-d})+(1-{(j+1)}^{-d})(1-{(j+2)}^{-d})=1-{(j+1)}^{-d}.
\]
Now, we suppose \eqref{eq: sum tp=1} is true for $M$ and we prove it is true also for $M+1$.
\[
\begin{split}
&\sum_{k=j+2}^{M+1} k^{-d}\prod_{m=j+1}^{k-1}(1-m^{-d})+\prod_{m=j+1}^{M+1}(1-m^{-d})\\
&=\sum_{k=j+2}^M k^{-d}\prod_{m=j+1}^{k-1}(1-m^{-d})+{(M+1)}^{-d}\prod_{m=j+1}^M(1-m^{-d})+{(M+1)}^{-d}\prod_{m=j+1}^M(1-m^{-d})\\
&=(1-{(j+1)}^{-d})({(M+1)}^{-d}+1-{(M+1)}^{-d})=1-{(j+1)}^{-d}.
\end{split}
\]
Therefore $p_{n|n-1}=1$ and the proof is completed.

\end{proof}
\section*{Acknowledgements}
This paper was written while the first author was a visiting professor at the Department of Decision Sciences, Bocconi University, Milan, Italy. He is grateful to his hosts for their hospitality and the extremely constructive atmosphere.

\bibliographystyle{chicago}
\bibliography{evt}
\end{document}